\documentclass[notitlepage,11pt,reqno]{amsart}
\usepackage[foot]{amsaddr}
\usepackage[numbers,sort]{natbib}

\usepackage{amssymb,nicefrac,bm,upgreek,mathtools,verbatim,esint}
\usepackage[final]{hyperref}
\usepackage[mathscr]{eucal}
\usepackage{dsfont}
\usepackage{graphicx}
\usepackage{caption}

\usepackage{color}
\definecolor{dmagenta}{rgb}{.4,.1,.5}
\definecolor{dblue}{rgb}{.0,.0,.5}
\definecolor{mblue}{rgb}{.0,.0,.8}
\definecolor{ddblue}{rgb}{.0,.0,.4}
\definecolor{dred}{rgb}{.6,.0,.0}
\definecolor{dgreen}{rgb}{.0,.5,.0}
\definecolor{Eeom}{rgb}{.0,.0,.5}

\usepackage[normalem]{ulem}

\usepackage[margin=1in]{geometry}
\allowdisplaybreaks

\newtheorem{lemma}{Lemma}[section]
\newtheorem{theorem}{Theorem}[section]

\newtheorem{corollary}{Corollary}[section]

\theoremstyle{definition}
\newtheorem{definition}{Definition}[section]

\theoremstyle{remark}

\newtheorem{remark}{Remark}[section]

\numberwithin{equation}{section}

\hypersetup{
  colorlinks=true,
  citecolor=dblue,
  linkcolor=dblue,
  frenchlinks=false,
  pdfborder={0 0 0},
  naturalnames=false,
  hypertexnames=false,
  breaklinks}

\newcommand{\R}{\mathbb{R}} 
\newcommand{\Prob}{\mathbb{P}} 

\newcommand{\dist}{\textnormal{dist}} 
\newcommand{\supp}{\textnormal{supp}} 
\newcommand{\essinf}{\textnormal{essinf}} 
\newcommand{\tr}{\textnormal{tr}} 

\renewcommand{\phi}{\varphi}

\newcommand{\cA}{{\mathcal A}}

\newcommand{\cC}{{\mathcal C}}

\newcommand{\Psidel}{\Psi(-\Delta)}

\newcommand{\eps}{\varepsilon}
\newcommand{\Ex}{\mathbb{E}}

\renewcommand{\tr}{\textnormal{Tr}} 
\begin{document}

\title[Overdetermined problems for nonlocal operators]{On overdetermined Problems for a general class of nonlocal operators}

\author{Anup Biswas}
\address{ Department of Mathematics, Indian Institute of Science Education and Research, Dr. Homi Bhabha Road, Pashan, Pune 411008, India}
\email{anup@iiserpune.ac.in}

\author{Sven Jarohs}
\address{Institut f\"ur Mathematik, Goethe-Universit\"at, Frankfurt, Robert-Mayer-Stra\ss e 10, D-60054 Frankfurt, Germany}
\email{jarohs@math.uni-frankfurt.de}

\begin{abstract}
We study the overdetermined problem for a large family of non-local operators given by generators of subordinate Brownian motions.
In particular, this family includes the fractional Laplacian, relativistic stable operators etc.
We consider these problems in bounded domains, exterior domains, and in annular domains and we show that under suitable conditions, the domains
and solutions are both radially symmetric. Our method uses both analytic and probabilistic tools.
\end{abstract}



\maketitle

\section{Introduction}
In his celebrated work \cite{S71}, Serrin solved the following overdetermined problem: Given a bounded $\cC^2$ domain $D$, if there exists 
a positive solution $u$ to
\begin{equation}\label{S-OD}
-\Delta u =1 \quad \text{in} \; D,\quad u=0\quad \text{on}\; \partial D, \quad \partial_\eta u = c\in\R \quad \text{on}\; \partial D,
\end{equation}
then $D$ is necessarily a ball. In the above, $\partial_\eta$ denotes the Neumann derivative along the exterior normal $\eta$. A very large number of extensions of Serrin’s result can be found in the literature and it is virtually impossible to provide a complete list of bibliography. To cite a few we refer to
\cite{AM11,BD13,LS07,FK08,CS09b,FV10a,FV10b,SS15}. Overdetermined
problems for the Laplacian operator in exterior domains are first studied by Reichel \cite{R97}. In this
work it was established that 
if $D$ and $\R^d\setminus \bar{D}$ are connected, $D$ is a bounded $\cC^2$ domain, and there exists a solution to
\begin{align*}
-\Delta u &= f(u) \quad \text{in}\; \R^d\setminus \bar{D}, \quad u\to 0, \quad \text{as}\; |x|\to\infty,
\\
u&= A\quad \text{on}\; \partial D, \quad \partial_n u=\mathrm{constant}\leq 0 \quad \text{on}\; \partial D,
\quad 0\leq u< A \quad \text{in}\; \R^d\setminus \bar{D},
\end{align*}
where $f$ is Lipschitz in $[0, A]$ and non-increasing close to $0$, then $D$ has to be a ball, and $u$ is radially symmetric and radially decreasing.
This result is further extended to quasi-linear operators in \cite{R96, R97}. Overdetermined problems for annular domains are considered in \cite{A92,P90,R95}.
The key ingredients in all the above works are a boundary point lemma (or Hopf's lemma) and the moving plane method. Very recently, overdetermined problems for the fractional
Laplacian have been studied. It should be kept in mind that for the $\alpha$-fractional Laplacian operator the 
Neumann derivative in \eqref{S-OD} does not exist due to the boundary behavior of solutions, but can be replaced by
$$(\partial_\eta)_\alpha u = \lim_{t\to 0+}\frac{u(x-t \eta)}{t^\alpha}\quad \text{for}\; x\in \partial D.$$
Fall and the second named author \cite{FJ13} (see also \cite{DG13} for $d=2, \alpha=\frac{1}{2}$) study the overdetermined problem for the fractional Laplacian
in bounded domains whereas Soave and Valdinoci \cite{SV14} consider the problem in exterior and annulus domains. Let us also mention \cite{GS16} where an overdetermined problem for the fractional Laplacian
is studied with $(\partial_n)_\alpha u$ being given by a suitable increasing function on the boundary.

In the present article we generalize the above results to a large family of isotropic nonlocal
operators. More precisely, these operators are obtained as the generator of subordinate Brownian motions.
For instance, when the subordinator is a $\alpha$-stable process we get the $\alpha$-fractional
Laplacian as the generator of the corresponding subordinate Brownian motion. We refer to 
Section~\ref{Sec-SBM} for more details. In this article, we denote these operators by $\Psidel$,
where $\Psi$ is a Bernstein function vanishing at $0$, and given by
$$\Psidel f(x) = \int_{\R^d} (f(x)-f(x+z)+1_{\{|z|\}\leq 1\}} z\cdot \nabla f(x)) j(|z|) dz,
\quad f\in\cC^2_b(\R^d),$$
where $j$ is a suitable nonlocal kernel (see \eqref{E2.3}). For $\Psi(t)=t^{\alpha}, \alpha\in (0, 1)$,
we have $j(r)=r^{-d-2\alpha}$ for $r>0$. Note that unlike the fractional Laplacian, these operators need
not have a global scaling property. It turns out that if $\Psi$
has certain \textit{scaling properties} at $0$ and $\infty$ then one can obtain the heat kernel
estimates for the associated operator and therefore, a PDE analysis is possible in many cases.
Interested readers may consult \cite{BGR14b,KSV} for more details. As mentioned above, one of the key steps 
in studying overdetermined problems for $\Psidel$ is to have Hopf-type lemmas and narrow
domain maximum principles. The maximum principle may be obtained as consequence of heat-kernel
estimates (see \cite{BL17} where the authors consider semigroup solutions). Using the 
observation that any viscosity solution can be represented as a semigroup solution we obtain
the narrow domain maximum principle (see Theorem~\ref{thm:anti-mp}). Recently, the first named author and L\H{o}rinczi \cite{BL19}
establish a Hopf's lemma for $\Psidel$. For this they use the sharp boundary behavior for Dirichlet solutions of $\Psidel$ obtained in \cite{KKLL18}. It turns out that the 
renewal function $V$ corresponding to a one dimensional L\'evy process associated to the one generated by $\Psidel$ is related to the
boundary behavior of the Dirichlet solution. This can be heuristically seen as follows: for the
domain $D$ if we consider the Dirichlet problem $\Psidel v=1$ in $D$ with vanishing exterior condition,
then the unique solution is given by the mean exit time of the L\'evy process from $D$. Therefore, 
applying \cite[Theorem~4.6]{BGR15}, it follows that $v\approx V(\delta_D)$ where $\delta_D$ denotes the distance function from the boundary $\partial D$.
 We draw our inspiration from these results to find
a Hopf's lemma for anti-symmetric supersolutions (see Theorem~\ref{thm:anti-mp2}) as also a corner point version of Hopf's Lemma (see Lemma \ref{lemma:cornerpoint-hopf}). These ingredients
play a key role in our overdetermined problems and the application of the moving plane method. We study the overdetermined problems in bounded
domains (Theorem~\ref{thm2.1} and ~\ref{thm2.2}), exterior domains (Theorem~\ref{thm2.3} and ~\ref{thm2.3a}), and in annuli domains (Theorem~\ref{thm2.3b}).

\subsection{Subordinate Brownian motions}\label{Sec-SBM}
The class of non-local operators we would be interested in are generators of a large family of L\'evy processes,
known as subordinate Brownian motions. These processes are obtained by a time change of a Brownian motion  by independent 
subordinators.
In this section we briefly recall the essentials of the subordinate process which will be particularly used in this article.

A Bernstein function is a non-negative completely monotone function, i.e., an element of the set
$$
\mathcal B = \left\{f \in \cC^\infty((0,\infty)): \, f \geq 0 \;\; \mbox{and} \:\; (-1)^n\frac{d^n f}{dx^n} \leq 0,
\; \mbox{for all $n \in \mathbb N$}\right\}.
$$
In particular, Bernstein functions are increasing and concave. We will consider the following subset
$$
{\mathcal B}_0 = \left\{f \in \mathcal B: \, \lim_{x\downarrow 0} f(x) = 0 \right\}.
$$
 For a detailed discussion of Bernstein functions we refer to the monograph \cite{SSV}.
Bernstein functions are closely related to subordinators. Recall that a subordinator $\{S_t\}_{t\geq 0}$
is a one-dimensional, non-decreasing L\'evy
process defined
 on some
probability space $(\Omega_S, {\mathcal F}_S, \mathbb P_S)$ .
 The Laplace transform of a subordinator is given by a
Bernstein function, i.e.,
\begin{equation*}
\label{lapla}
\Ex_{\mathbb P_S} [e^{-xS_t}] = e^{-t\Psi(x)}, \quad t, x \geq 0,
\end{equation*}
 where $\Psi \in {\mathcal B}_0$. In particular, there is a bijection between the set of subordinators on a given
probability space and Bernstein functions with vanishing right limits at zero.

Let $ B$ be an $\R^d$-valued Brownian motion on the Wiener space $(\Omega_W,{\mathcal F}_W, \mathbb P_W)$, running twice
as fast as standard $d$-dimensional Brownian motion, and let $ {S}$ be an independent subordinator with characteristic
exponent $\Psi$. The random process
$$
\Omega_W \times \Omega_S \ni (\omega_1,\omega_2) \mapsto B_{S_t(\omega_2)}(\omega_1) \in \R^d
$$
is called subordinate Brownian motion under $ {S}$. For simplicity, we will denote a subordinate Brownian motion
by $ \{X_t\}_{t\geq 0}$, its probability measure for the process starting at $x \in \R^d$ by $\mathbb P_x$, and expectation with respect
to this measure by $\Ex_x$. Note that the characteristic exponent of a pure jump process $\{X_t\}_{t\geq 0}$ (i.e., with $b=0$) is given
by
\begin{equation*}
\Psi(|z|^2)= \int_{\R^d\setminus\{0\}} (1-\cos(y\cdot z)) j(|y|) \, d{y},
\end{equation*}
where the L\'evy measure of $\{X_t\}_{t\geq 0}$ has a density $y\mapsto j(|y|)$, $j:(0, \infty)\to (0, \infty)$, with respect to
the Lebesgue measure, given by
\begin{equation}\label{E2.3}
j(r) = \int_0^\infty (4\pi t)^{-d/2} e^{-\frac{r^2}{4t}} \, \mathfrak m (d{t}),
\end{equation}
where $\mathfrak m$ is the unique measure on $(0, \infty)$ satisfying
$$\Psi(\lambda)=\int_{(0, \infty)} (1-e^{-\lambda t}) \mathfrak m (d{t}).$$
In this article we impose the following \textit{weak scaling condition} on the subordinators.
\begin{equation}\label{assumption1}
\text{There are $0<a_1\leq a_2<1\leq b_1$ such that}\quad \frac{1}{b_1}\Big(\frac{R}{r}\Big)^{a_1}\leq \frac{\Psi(R)}{\Psi(r)}\leq b_1\Big(\frac{R}{r}\Big)^{a_2}\quad\text{for $1\leq r\leq R$},
\end{equation}
and,
\begin{equation}\label{assumption2}
\text{there is $b_2>1$ such that}\quad j(r)\leq b_2\,j(r+1)\quad\text{for $r\geq 1$.}
\end{equation}
There is large family of subordinators that satisfy \eqref{assumption1} (see \cite{BL17,KKLL18}).
Moreover, any complete Bernstein function satisfying \eqref{assumption1} also satisfies 
\eqref{assumption2} \cite[Theorem~13.3.5]{KSV}.

For some of our proofs below we  use some information on the normalized ascending ladder-height process of
$ \{X^1_t\}_{t\geq 0}$, where $X^{1}_t$ denotes the first coordinate of $X_t$. Recall that the ascending ladder-height
process of a L\'evy process $ Z$ is the process of the right inverse $(Z_{L^{-1}_t})_{t\geq 0}$, where $L_t$
is the local time of $Z_t$ reflected at its supremum (for details and further information we refer to \cite[Chapter 6]{B96}).
Also, we note that the ladder-height process of $X^1$ is a subordinator with Laplace exponent
$$
\tilde\Psi(x)=\exp\left(\frac{1}{\pi}\int_0^\infty \frac{\log \Psi(xy)}{1+y^2}\, d{y}\right), \quad x \geq 0.
$$
Consider the potential measure $V(x)$ of this process on the half-line $(-\infty, x)$. Its Laplace transform is
given by
$$
\int_0^\infty V(x) e^{-sx}\, d{x}= \frac{1}{s\tilde\Psi(s)}, \quad s > 0.
$$
It is also known that $V(x)=0$ for $x\leq 0$, the function $V$ is continuous and strictly increasing in $(0, \infty)$
with $V(\infty)=\infty$ (see \cite{F74} for more details). As shown in \cite[Lemma~1.2]{BGR14} and \cite[Corollary~3]{BGR14b},
there exists a constant $C = C(d)$ such that
\begin{equation*}
\frac{1}{C}\,{\Psi(1/r^2)}\leq \frac{1}{V^2(r)}\leq C\, {\Psi(1/r^2)}, \quad r>0.
\end{equation*}
Also, using \cite[see expression (15)]{BGR14b} we obtain
\begin{equation}\label{J-V}
j(r) \leq C \frac{\Psi(r^{-2})}{r^d}\quad r>0,
\end{equation}
for some constant $C>1$. 
We define the operator
\begin{align*}
-\Psidel f(x) &= \int_{\R^d} (f(x+z)-f(x)-1_{\{|z|\}\leq 1\}} z\cdot \nabla f(x)) j(|z|) dz
\\
&=\frac{1}{2} \int_{\R^d} (f(x+z)+f(x-z)-2 f(x)) j(|z|) dz,
\end{align*}
for $f\in\cC^2_{b}(\R^d)$, by functional calculus.
 The operator $-\Psidel$ is the Markov generator of
subordinate Brownian motion $\{X_t\}_{t\geq 0}$ corresponding to the subordinator $S$, uniquely determined by $\Psi$.


\section{Main results}\label{S-main}

Let $D\subset \R^d, d\geq 2,$ be a bounded domain with $\cC^2$ boundary. Now given any continuous 
function $f$ let us consider the viscosity solution $u\in\cC(\bar D)$ of
\begin{equation*}
\Psidel u = f\quad \text{in}\; D, \quad \text{and,}\quad u=0\quad \text{in}\; D^c.
\end{equation*}
By \cite[Theorem~2.1]{KKLL18} and \eqref{assumption1} we know that $u$ is H\"{o}lder continuous in $\R^d$. Furthermore,
the function $u/V(\delta_D)$ is H\"{o}lder continuous in $D$ \cite[Theorem~2.2]{KKLL18} where
$\delta_D(x)=\inf d(x, D^c)$ is the distance function from the boundary. Thus, we can define the
\text{trace} of $u/V(\delta_D)$ on the boundary $\partial D$ as
\begin{equation}\label{def:trace}
\tr_V(u)(x)=\lim_{D\ni z\to x} \frac{u(x)}{V(\delta_D(z))}, \quad x\in \partial D.
\end{equation}
This trace operator plays an important role in the study of overdetermined problems. 

Let $B_r$ denote the ball of radius $r$ around $0$. By $\tau_r$ we denote the exit time of $X$
from $B_r$ i.e.
$$\tau_r=\inf\{t>0\; :\; X_t\notin B_r\}.$$
By \cite[Section 4.2]{BL19} the function $u_r(x)=\Ex_x[\tau_r]$ is a solution
to the \textit{overdetermined} problem
\begin{equation}\label{eq:ball}
\left\{\begin{aligned}
\Psi(-\Delta)u_r &=1&&\text{in $B_r(0)$;}\\
u_r &=0 &&\text{in $\R^d\setminus B_r(0)$;}\\
 \tr_V(u_r)&={\rm constant}=H_r &&\text{on $\partial B_r(0)$.} 
\end{aligned}\right.
\end{equation}
Furthermore, the function $r\mapsto H_r$ is positive and strictly increasing in $(0, \infty)$ \cite[Lemma~4.1]{BL19}.
Here and in what follows, by a solution we always mean a viscosity solution.

The goal of this work is to show that indeed for an open bounded set $D$ the \textit{overdetermined problem}
\begin{equation}\label{eq:overdetermined}
\left\{\begin{aligned}
\Psi(-\Delta)u&=1&&\text{in $D$;}\\
u&=0 &&\text{in $\R^d\setminus D$,} 
\end{aligned}\right.
\end{equation}
with $\tr_V(u)={\rm constant}$ on $\partial D$ has a solution if and only if $D=B_r(0)$, that is, we have
\begin{theorem}\label{thm2.1}
Let $D$ be an open bounded set with $\cC^2$ boundary and $d\geq 2$. Assume there is a solution $u$ of \eqref{eq:overdetermined} satisfying for some fixed $c\in \R$
\begin{equation}\label{corrected assumption}
\frac{u}{V(\delta_D)}\in C^2(\overline{D})\quad\text{and}\quad \tr_V(u)=c \quad\text{ on $\partial D$.}
\end{equation}
 Then, up to translation, $D=B_r(0)$ for some $r>0$ and $u=u_r$ given by \eqref{eq:ball}.
\end{theorem}

\begin{remark}
We emphasize that connectedness of $D$ is not assumed a priori. In fact, due to the nonlocal character of $\Psi(-\Delta)$ the connectedness follows a posteriori.
\end{remark}

Theorem \ref{thm2.1} is indeed a special case of a more general result concerning semilinear problems.
In particular, we show

\begin{theorem}\label{thm2.2}
Let $D\subset\R^d, d\geq 2,$ be an open bounded set with $\cC^2$ boundary. 
Let $f\in \cC^{0,1}(\R)$ and assume there is a nonnegative nontrivial bounded solution of
\begin{equation*}
\left\{\begin{aligned}
\Psi(-\Delta)u&=f(u)&&\text{in $D$,}\\
u&=0 &&\text{in $\R^d\setminus D$,} 
\end{aligned}\right.
\end{equation*}
which satisfies for some fixed $c\geq0$
$$
	\frac{u}{V(\delta_D)}\in C^2(\overline{D})\quad\text{and}\quad \tr_V(u)=c \quad\text{ on $\partial D$.}
$$
 Then, up to translation, $D=B_r(0)$ for some $r>0$, $u>0$ in $D$, and $u$ is radially symmetric and strictly decreasing in the radial direction.
\end{theorem}

Our next result concerns an overdetermined problem in the complement of a bounded set in the spirit of \cite{SV14} (for the local case, we also refer to \cite{R97}). Note that here, we assume $u$ to be a positive constant on this compact set. Hence the trace defined in \eqref{def:trace} is adjusted by this constant.

\begin{theorem}\label{thm2.3}
	Let $G_1,\ldots,G_n\subset \R^d, d\geq 2,$ be a family of compact connected sets with $\cC^2$ boundary such that $G_i\cap G_j=\emptyset$ for $i\neq j$ and let $G:=\bigcup_{k=1}^n G_k$. Let $A>0$ and $f\in C^{0,1}([0,A])$ be nonincreasing for small arguments. Assume that for some given $A_1,\ldots,A_n\leq 0$ there is a solution $u\in C^{\beta}(\R^d)$ for some $\beta\in(0,1)$ of the problem 
	\begin{equation*}
	\left\{\begin{aligned}
	\Psi(-\Delta)u&=f(u)&&\text{in $\R^d\setminus G$,}\\
	u&=A &&\text{in $G$,}\\
	\tr_V(u-A)&=A_k && \text{on $\partial G_k$},
	\end{aligned}\right.
	\end{equation*}
	with $0\leq u< A$ on $\R^d\setminus G$, $\lim\limits_{|x|\to\infty} u(x)=0$, and
	$$
	\frac{u}{V(\delta_G)}\in C^2(\overline{G})
	$$
	Then $G$ is a ball and $u$ is radially symmetric and strictly decreasing in the radial direction with respect to the center of $G$.
\end{theorem}

We emphasize that in contrast to the \cite[Theorem 1.3]{SV14}, where the above theorem was stated for the fractional Laplacian, we do not need an additional regularity assumption on $u/(V\circ\delta_{\R^d\setminus G})$. Indeed, the above theorem extends the one from \cite{SV14} and this is mainly due a narrow type maximum principle Lemma \ref{lemma:mp-unbounded}.

\begin{remark}\label{R2.2}
As we have discussed above the trace operator of the boundary is justified from \cite{KKLL18},
but when $D$ is unbounded, some explanation is required to justify the boundary trace $\tr_V(\cdot)$.
To do so consider the exterior domain Dirichlet problem
\begin{equation*}
\Psidel u = f\quad\text{in}\; G^c, \quad \text{and}\quad u=0\quad \text{in} \; G,
\end{equation*}
where $f$ is a bounded continuous function, $G$ is closed bounded with $\cC^2$ boundary, and $u$ is a bounded
viscosity solution. Observe that for any bounded domain $D\subset G^c$, with $\cC^2$ boundary,
we can write (see \cite{BL19} or \cite[Remark~3.2]{BL17})
$$u(x) = \Ex_x[u(X_{\tau_D})] - \Ex_x\left[\int_0^{\tau_D} f(X_s) ds\right],\quad x\in D,$$
where $\tau_D$ denotes the exit time from $D$. This follows from the uniqueness of viscosity solution
\cite[Theorem~3.8]{KKLL18} and the fact that the right hand side function is a viscosity solution
in $D$ with exterior boundary data being $u$. With this representation we can write $u=w_1+w_2$
where
$$w_1(x)=\Ex_x[u(X_{\tau_D})],\quad \text{a harmonic function in}\; D,$$
and,
$$w_2(x)=- \Ex_x\left[\int_0^{\tau_D} f(X_s) ds\right]
\quad \text{satisfying}\quad \Psidel w_2 =f \quad \text{in}\; D, \;\; w_2=0\; \; \text{in}\; D^c.$$
Furthermore, we also check that for any $x\in D$ we have 
\begin{equation}\label{ER2.2A}
\cA w_1(x):=\lim_{t\downarrow 0} \frac{1}{t}\left(\Ex_{x}[w_1(X_t)]-w_1(x)\right)=0.
\end{equation}
Indeed, for any ball $B_r(x)\subset D$ we have
\begin{align*}
\left|\Ex_{x}[w_1(X_t)]-w_1(x)\right|
&= \left|[w_1(X_t)]-\Ex_x[w_1(X_{t\wedge\tau_{B_r(x)}})]\right|
\\
&= \left|\Ex_x[w_1(X_t)1_{\{t> \tau_{B_r(x)} \}}]-\Ex_x[w_1(X_{\tau_{B_r(x)}})1_{\{t> \tau_{B_r(x)} \}} ]\right|
\\
&= \left|\Ex_x[1_{\{t> \tau_{B_r(x)} \}} \Ex_{X_{\tau_{B_r(x)}}}[w_1(X_{t-\tau_{B_r(x)}})]]-\Ex_x[w_1(X_{\tau_{B_r(x)}})1_{\{t> \tau_{B_r(x)} \}} ]\right|,
\end{align*}
where in the first and third line we use the strong Markov property of $X$. Now observe that for any non-negative cut-off function $\zeta$, $0\leq\zeta\leq 1$,
 that vanishes outside a 
compact set, we have $\zeta w$ uniformly continuous in $\R^d$, and therefore,
$$\sup_{x\in \R^d} \Ex_x[\sup_{s\leq t}|(\zeta w_1)(X_s)-\zeta(x)w_1(x)|]\to 0, \quad \text{as}\; t\to 0.$$
Thus using \cite[Corollary~2.8]{BGR15} we get 
$$\frac{1}{t}\left|\Ex_x[1_{\{t> \tau_{B_r(x)} \}} \Ex_{X_{\tau_{B_r(x)}}}[\zeta w_1(X_{t-\tau_{B_r(x)}})]]-\Ex_x[\zeta w_1(X_{\tau_{B_r(x)}})1_{\{t> \tau_{B_r(x)} \}} ]\right|
\to 0, $$ 
as $t\to 0$. Now let $\zeta=1$ in $B_m(0)$ and let $\tau_m$ be the exit time from $B_m(0)$. Then
\begin{align*}
\frac{1}{t} \Ex_x[|(1-\zeta(X_t))w_1(X_t)|1_{\{t> \tau_{B_r(x)} \}}] &\leq \frac{1}{t} \sup |w_1| \Prob_x(\tau_m\leq t)
\\
&\lesssim \sup |w_1|\, \frac{1}{V^2(m)}\to 0, \quad \text{as}\; m\to \infty,
\end{align*}
where the last line follows from \cite[Corollary~2.8]{BGR15}. Again, 
$$|(1-\zeta(X_{\tau_{B_r(x)}}))w_1(X_{\tau_{B_r(x)}})|1_{\{t> \tau_{B_r(x)} \}}\neq 0$$
would imply for some point $s\leq t$ we have $|X_s|> m$, and thus, it is included in $\{\tau_m\leq t\}$. So we can use the above argument to show
$$\frac{1}{t}\Ex_x\left[|(1-\zeta(X_{\tau_{B_r(x)}}))w_1(X_{\tau_{B_r(x)}})|1_{\{t> \tau_{B_r(x)} \}}\right]\to 0, \quad \text{uniformly in}\; t>0,$$
as $m\to\infty$. Thus we have \eqref{ER2.2A}. 

With the above decomposition of $u$ (i.e. $w_1+w_2$ with respect to any sub-domain $D$) one can follow the arguments of \cite{KKLL18} to conclude that
$\tr_V(u)$ exists on $\partial G^c$. Indeed, this will be a consequence of \cite[Lemma~4.10]{KKLL18}. Furthermore, if $u\in\cC^\beta(\R^d)$ then we get
from the arguments of Theorem~2.2 of \cite{KKLL18} that $\frac{u}{V(\delta_{G^c})}$ is H\"{o}lder continuous in a neighborhood of
$\partial G^c$.
\end{remark}

\begin{theorem}\label{thm2.3a}
The conclusion of Theorem \ref{thm2.3} remains true, if $f$ is additionally assumed to be non-increasing in $[0,A]$ but only assuming $0\leq u\leq A$.
\end{theorem}

In the spirit of the moving plane method, which we use to proof Theorem \ref{thm2.3}, we have the following result

\begin{theorem}\label{thm2.4}
	Let $A>0$ and $f\in C^{0,1}([0,A])$ such that $f$ is non-increasing for small $t$. Then any continuous solution $u$ of the problem 
	\begin{equation*}
	\left\{\begin{aligned}
	\Psi(-\Delta)u&=f(u)&&\text{in $\R^d\setminus\{0\}$,}\\
	u(0)&=A &&
	\end{aligned}\right.
	\end{equation*}
	with $0\leq u\leq A$ on $\R^d$ and $\lim\limits_{|x|\to\infty}u(x)=0$ is radially symmetric and strictly decreasing in the radial direction.
\end{theorem}

An immediate consequence gives
\begin{corollary}\label{cor-from-thm2.4}
	Let $f\in C^{0,1}((0,\infty))$ such that $f$ is non-increasing for small $t$. Then any nonnegative continuous solution $u$ of the problem 
	\begin{equation*}
	\Psi(-\Delta)u=f(u)\quad \text{in $\R^d$,}\quad \text{with}\quad \lim\limits_{|x|\to\infty}u(x)=0
	\end{equation*}
	 is radially symmetric up to translation. Moreover, either $u\equiv 0$ or up to translation $u$ is strictly decreasing in the radial direction.
\end{corollary}
 Concerning radial sets, we get
 
\begin{corollary}\label{cor-from-thm2.3}
	Let $r,A>0$, $A_1\in \R$, $d\geq2$, and $f\in C^{0,1}([0,A])$ be non-increasing. Then any 
	H\"{o}lder continuous solution $u$ of 
	\begin{equation*}
	\left\{\begin{aligned}
	\Psi(-\Delta)u&=f(u)&&\text{in $\R^d\setminus \overline{B_r(0)}$,}\\
	u&=A &&\text{in $\overline{B_r(0)}$,}\\
	\end{aligned}\right.
	\end{equation*}
	with $0\leq u\leq A$ and $\lim\limits_{|x|\to\infty}u(x)=0$ is radially symmetric and strictly decreasing in the radial direction.
\end{corollary}

Moreover, the approach in complements allows also to tread annular-like sets.

\begin{theorem}\label{thm2.3b}
	Let $G_1,\ldots,G_n\subset \R^d, d\geq 2,$ be a family of compact connected sets with $\cC^2$ boundary such that $G_i\cap G_j=\emptyset$ for $i\neq j$ and let $G:=\bigcup_{k=1}^n G_k$. Let $D\subset \R^d$ be an open bounded set with $C^2$ boundary and such that $G\Subset D$. Let $A>0$ and $f\in C^{0,1}([0,A])$. Assume that for some given $B\geq0$, $A_1,\ldots,A_n\leq 0$ there is a continuous solution $u$ of the problem 
	\begin{equation*}
	\left\{\begin{aligned}
	\Psi(-\Delta)u&=f(u)&&\text{in $D\setminus G$,}\\
	u&=0 &&\text{in $\R^d\setminus D$,}\\
	u&=A &&\text{in $G$,}\\
	\tr_V(u-A)&=A_k && \text{on $\partial G_k$},\\
	\tr_V(u)&=B &&\text{on $\partial D$},
	\end{aligned}\right.
	\end{equation*}
	with $0\leq u< A$ on $\R^d\setminus G$, and
	$$
	\frac{u}{V(\delta_{D\setminus G})}\in C^2(\overline{D\setminus G})
	$$
	Then $D$ and $G$ are concentric balls and $u$ is radially symmetric and strictly decreasing in the radial direction with respect to the center of $G$.
\end{theorem}

\begin{remark}
Consider a smooth function $\zeta$ satisfying the boundary condition 
$$\zeta=0\quad \text{in}\; D^c, \quad \text{and}\quad \zeta=A\quad \text{in} \; G.$$
$\zeta$ being smooth we have $g=\Psidel\zeta$ continuous in $D\setminus G$, and using \cite[Lemma~5.8]{CS09} we obtain
$$\Psidel (u-\zeta)=f(u)-g\quad \text{in}\; D\setminus G, \quad \text{and}\quad (u-\zeta)=0\quad \text{in}\; (D\setminus G)^c.$$
Hence, by \cite[Theorem~2.1]{KKLL18}, $(u-\zeta)$ is H\"{o}lder continuous in $\R^d$, and thus, $u$ is H\"{o}lder continuous in $\R^d$. Furthermore,
$\tr_V(u-\zeta)$ is also H\"{o}lder continuous. Note that since $\zeta$ is smooth, we have
$$\tr_V(\zeta)=0\quad \text{on}\; \partial D, \quad \text{and}\quad \tr_V(A-\zeta)=0\quad \text{on}\; \partial G_k.$$
This follows from the fact that $V(r)\gtrsim r^{a_2}$ for all $r\in (0, 1)$. Hence the operator $\tr_V$ is well defined on $u$ near the boundary. Furthermore, Remark~\ref{R2.2} suggests that the trace functions
are H\"{o}lder continuous near the boundaries.
\end{remark}

As before, as an immediate consequence, we have

\begin{corollary}\label{cor-from-thm2.3b}
	Let $0<r<R<\infty$, $A>0$, $A_1\in \R$, $d\geq2$, and $f\in C^{0,1}([0,A])$. Then any continuous solution $u$ of 
	\begin{equation*}
	\left\{\begin{aligned}
	\Psi(-\Delta)u&=f(u)&&\text{in $B_R(0)\setminus \overline{B_r(0)}$;}\\
	u&=0 &&\text{in $\R^d\setminus B_R(0)$;}\\
	u&=A &&\text{in $\overline{B_r(0)}$;} 
	\end{aligned}\right.
	\end{equation*}
	with $0\leq u\leq A$ is radially symmetric and strictly decreasing in the radial direction.
\end{corollary}

\begin{remark}
	We note that an assumption as in \eqref{corrected assumption} has been used in the proof of Lemma \ref{lemma:diagonal-decay} as part of studying the corner point case. However, in the previous publications it was missing in the statements of the Theorems \ref{thm2.1}, \ref{thm2.2}, \ref{thm2.3}, \ref{thm2.3a}, and \ref{thm2.3b}. We emphasize that this assumption is on the contrary not needed in the proofs of Theorem \ref{thm2.4} and Corollaries \ref{cor-from-thm2.4}, \ref{cor-from-thm2.3}, and \ref{cor-from-thm2.3b}.
\end{remark}


\section{Proofs of main results}

\subsection{Maximum principles for anti-symmetric viscosity solutions}
As is well known, maximum principles for anti-symmetric solutions play a key role in the analysis
of overdetermined problems. In this section we develop all the required tools in the direction
which will later be used to establish our main results. Since our framework is based on viscosity solutions, we recall the definition of viscosity solution from \cite{CS09} for convenience.

\begin{definition}[Viscosity solution]
Let $D\subset \R^d$ be open. An upper semi-continuous function $u:\R^d\to\R$ in $\bar D$ is said to be a
\emph{viscosity sub-solution} of
\begin{equation}
\label{visco}
-\Psidel u + c(x) u =  g\quad \text{in}\; D,
\end{equation}
if for every $x\in D$ and test function $\varphi\in \cC_{\rm b}(x)$ ($\cC_{\rm b}(x)$ is the set of
all bounded continuous functions that are twice continuously differentiable in a neighborhood of $x$)
satisfying $u(x)=\varphi(x)$ and
$$
\varphi(y) > u(y) \quad y\in \R^d\setminus\{x\},
$$
we have
$$
-\Psidel \varphi(x) + c(x) u(x) \geq  g(x).
$$
Similarly, a lower semi-continuous function is a \emph{viscosity super-solution} of \eqref{visco} whenever
$\varphi(y) < u(y)$, $y\in \R^d\setminus\{x\}$, implies $-\Psidel \varphi(x) + c(x) u(x) \leq  g(x)$.
Furthermore, $u$ is said to be a \emph{viscosity solution} if it is both a viscosity sub- and super-solution. 
\end{definition}

Let $H$ be the half-space defined by 
$$
H=\{x\in\R^d\; :\; x\cdot e>\lambda\},
$$
and we denote by $x_{\lambda}=R(x)=R_{\lambda,e}(x):=x-2(x\cdot e)e+2\lambda e$ the reflection at $\partial H=\{x\cdot e=\lambda\}$ of $x$. Let $u:\R^d\to \R$, then we call $u$ \textit{anti-symmetric (w.r.t. $H$)} if $u=-u\circ R$ on $\R^d$. Let $\Omega\subset H$ and $u$ be a bounded anti-symmetric super-solution of $\Psidel u= g$ in $\Omega$ with $u\geq 0$ in $H\setminus\Omega$.
Define 
\begin{equation}\label{defi:function-for-anti-mp}
v=\left\{\begin{array}{lll}
-u & \text{if} \; x\in \{u<0\} \cap \Omega
\\
0 & \text{otherwise}.
\end{array}
\right.
\end{equation}
Denote by $\Sigma= \{u<0\} \cap \Omega$. Note that $v$ is upper semi-continuous and $\Sigma$ is open.
We claim that 
\begin{equation}\label{A3}
-\Psidel v \geq g \quad \text{in}\; \Sigma,
\end{equation}
in viscosity sense. To prove the above, consider any point $x\in\Sigma$ and a test function $\varphi$ that touches $v$ from above at the point $x$.
Then $\psi:=\varphi + (-u-v)$ touches $-u$ from above at the point $x$. Since $-u$ is a sub-solution, it follows that
$$-\Psidel\varphi(x) -\Psidel\eta(x) =-\Psidel \psi(x) \geq  g(x), $$
where $\eta=-u-v$. To prove \eqref{A3}, we only need to show that 
\begin{equation*}
-\Psidel\eta(x)\leq 0.
\end{equation*}
Since $\eta=0$ in $\Sigma$ and equals to $-u$ in $\Sigma^c$, we have
\begin{align*}
-\Psidel\eta(x) &= -\int_{\Sigma^c} u(z) j(|x-z|) dz
\\
&= -\int_{R(\Sigma)} u(z) j(|x-z|) dz - \int_{R(H\setminus \Sigma)\cap(H\setminus \Sigma) }u(z) j(|x-z|) dz
\end{align*}
Note that the first term is non-positive due to the anti-symmetry of $u$ and the fact $u<0$ in $\Sigma$. On the other hand the 
second term equals to 
$$
\int_{R(H\setminus \Sigma)\cap(H\setminus \Sigma) }u(z) j(|x-z|) dz
=\int_{H\setminus \Sigma}u(z) (j(|x-z|)-j(|x-R(z)|) dz\geq 0,$$
since $u\geq 0$ in $H\setminus \Sigma$, and $j$ is radially decreasing. This completes the proof of \eqref{A3}.
\eqref{A3} will be useful in proving 
anti-symmetric maximum principle.

\begin{theorem}\label{thm:anti-mp}
Let $H$ be a half-space, $\Omega\subset H$ open and bounded, and $c$ be bounded. Then there exists $p\geq 1$ and $C>0$,
depending only on $\Psi$ and diameter of $\Omega$, such that if $u\in\cC_{b}(\R^d)$ is an anti-symmetric super-solution of 
$$\Psidel u + c(x) u=0 \quad\text{in $\Omega$,}$$
with $u\geq0$ on $H\setminus \Omega$, then we must have 
\begin{equation*}
\sup_{\Omega} u^-\leq C \|c^+\|_{L^\infty(\Omega)}\,\|u^-\|_{L^p(\Omega)}.
\end{equation*}
In particular, given $c_{\infty}>0$ such that $c^+\leq c_{\infty}$ on $\Omega$ there is $\delta>0$ such that if $|\Omega|<\delta$, then $u\geq 0$.
\end{theorem}

\begin{proof}
As discussed above, we consider the set $\Sigma=\{u<0\}\cap \Omega$ and the function $v$ as in \eqref{defi:function-for-anti-mp}. From \eqref{A3} we have
\begin{equation*}
-\Psidel v + c^+(x) v \geq 0 \quad \text{in}\; \Sigma.
\end{equation*}
We follow the arguments of \cite[Theorem~3.1]{BL19}.
Since $\partial\Sigma$ is not nice in general, we consider a collection of increasing smooth sets $\{D_n\}_n$, contained in $\Sigma$ and
increasing to $\Sigma$. Let $w_n$ be the unique viscosity solution of
$$-\Psidel w_n =-g\quad \text{in}\; D_n, \quad \text{and}\quad w_n=v\quad \text{in} \; D^c_n,$$
where $g=\|c^+\|_{L^\infty(\Omega)} v$. From the comparison principle \cite[Theorem~5.2]{CS09}, \cite[Theorem~3.8]{KKLL18} it follows that
$$v\leq w_n\quad \text{in}\; \R^d.$$
Moreover, $w_n$ attends a stochastic representation given by
\begin{equation}\label{ET3.1B}
w_n(x)= \Ex_x[v(X_{\tau_{D_n}})] + \Ex_x\left[\int_0^{\tau_{D_n}} g(X_s) ds\right], \quad x\in D_n.
\end{equation}
Hence, using \eqref{assumption1}-\eqref{ET3.1B} and \cite[Theorem~3.3]{BL17} we see that for some
constant $C, p>1$ we have
$$\sup_{D_n^c} w_n + C\|g\|_{L^p(D_n)}\leq \sup_{D_n^c} |v| + C\|g\|_{L^p(\Omega)},$$
which in turn, implies
$$\sup_{D_n} v \leq \sup_{D_n^c} |v| + C\|g\|_{L^p(\Omega)}.$$
Thus, the result follows by letting $n\to\infty$.
\end{proof}

Our next result concerns a version of Hopf's Lemma for $\Psi(-\Delta)$ for anti-symmetric super-solutions

\begin{theorem}\label{thm:anti-mp2}
Let $H$ be a half-space, $\Omega\subset H$ open, and $c\in L^\infty(\Omega)$. If $u\in\cC_b(\R^d)$ is an anti-symmetric super-solution of 
$$\Psidel u + c(x) u=0 \quad\text{in $\Omega$}$$
with $u\geq0$ in $H$. Then either $u\equiv 0$ or $u>0$ in $\Omega$. Moreover, if $u\not\equiv 0$ and there is $x_0\in \partial\Omega\setminus \partial H$ with $u(x_0)=0$ and such that there is a ball $B\subset \Omega$ with $x_0\in \partial B$, then there is $c>0$ such that 
\begin{equation*}
\liminf_{t\to 0^+} \frac{u(x_0-t\eta)}{V(t)}\geq c.
\end{equation*}
In particular, if $\tr_V(u)(x_0)$ exists, then $\tr_V(u)(x_0)>0$. Here $\eta$ denotes the outward normal at $x_0$.
\end{theorem}

\begin{proof}
Assume $u\not\equiv 0$ on $\R^d$, then there is a compact set $K\Subset H$ with $\inf_{K} u=\delta>0$.
Suppose that
$u(x)=0$ for some $x\in\Omega$. Consider a test function $\varphi\geq 0$ with the following property: for some ball $B_{2\delta}(x)\Subset \Omega$ we have
$$\varphi\leq u \;\text{in}\; \R^d, \quad \varphi=0\; \; \text{in}\; B_\delta(x), \quad \varphi=u\; \; \text{in}\; B^c_{2\delta}(x).$$
We may also choose $\delta$ small enough so that the ball is far from $K$.
Thus, by definition, we have
$$\Psidel \varphi(x)\geq 0,$$
which implies
\begin{align*}
0\geq \int_{\R^d}\varphi(y) j(|x-y|) dy &=\int_{R(B_{2\delta}(x))} u(z) j(|x-y|) dy + \int_{H\setminus B_{2\delta}(x)} u(y) (j(|x-y|)-j(|x-y_\lambda|)) dy
\\
&=\int_{R(B_{2\delta}(x))} u(z) j(|x-y|) dy + \int_{K} u(y) (j(x-y|)-j(|x-y_\lambda|)) dy.
\end{align*}
Thus if we choose $\delta$ small enough the RHS of the above display is positive which leads to a contradiction. Hence we must have $u>0$ in $\Omega$. This proves the first part.

Now we prove the second part. Let $B$ be a ball that touches $\Omega$ at $x_0$ from inside. Let $\vartheta$ be the expected exit time from $B$.
Define $w=a(\vartheta-\vartheta\circ R)$ as before. It is straightforward to see that
$$\Psidel w\leq a C\quad \text{in}\; B,$$
for some constant $C$. To complete the proof we only need to show that for some $a_0>0$ we have $u\geq w$ in $B$ and then, the proof follows from \eqref{eq:ball}. 
To the contrary, assume that no such $a_0$
exists. Now it follows from \cite[Lemma~5.8]{CS09} that $v_a=u- w$ is an anti-symmetric super-solution of
$$\Psidel v_a= - (\|c\|_{L^\infty(\Omega)} u +aC)=-g\quad \text{in}\; B.$$
Let $x_a\in\mathrm{Arg min}_{B} v_a$. Since $\min_{\bar B} v_a<0$, it follows that $x_a\to \partial B$ as $a\to 0$. This also implies
$u(x_a)\leq  w(x_a)=a \vartheta(x_a)\to 0$ as $a\to 0$.
Now we choose a test function
$\varphi(\geq \min_{\bar B} v_a)$ that touches $v_a$ at $x_a$ from below and agrees with $v_a$ outside a ball $B'\Subset B$. By definition, it then follows that
$$\Psidel \varphi(x_a)\geq - (\|c\|_{L^\infty(\Omega)} u(x_a) + Ca)\to 0 \quad \text{as}\; a\to 0.$$
Let us now compute the LHS. Let $K$ be any compact set inside $B$ and we may assume that $x_a\notin K$.
\begin{align*}
\Psidel \varphi(x_a) &= \int_{\R^d} (\varphi(x_a)-\varphi(y)) j(|x_a-y|) dy
\\
&=\int_{H} (\varphi(x_a)-\varphi(y)) (j(|x_a-y|)-j(|x_a-y_\lambda|)) dy
\\
&\leq \int_{K} (\varphi(x_a)-\varphi(y)) (j(|x_a-y|)-j(|x_a-y_\lambda|)) dy
\\
&\leq (a\|\vartheta\|_{L^\infty(B)}-\min_K u )\int_{K}  (j(|x_a-y|)-j(|x_a-y_\lambda|)) dy <0,
\end{align*}
for all $a$ small. This is a contradiction. Thus we have the second part.
\end{proof}

Our last result on maximum principles for antisymmetric functions concerns unbounded sets, where we have a sign on the linear part given by $c$.

\begin{lemma}\label{lemma:mp-unbounded}
Let $H$ be a half-space, $\Omega\subset H$ open, and $c_{\infty},R>0$. Then there is $\delta>0$ such that the following holds. Let $c\in L^\infty(\Omega)$ with $c^-\leq c_{\infty}$ and such that there is $K\subset \Omega$ with $\Omega\setminus K \subset B_R(0)$ and $c\geq 0$ on $K$. If $u\in \cC_b(\R^d)$ is an anti-symmetric super-solution of
\[
\Psi(-\Delta)u+c(x)u=0 \quad \text{in $\Omega$}
\]	
with $u\geq0$ in $\Big(H\setminus \Omega\Big)\cup\Big(\Big(\Omega\cap\{x\in H\;:\; \dist(x,\partial H)\geq \delta\}\Big)\setminus K\Big)$
and $\liminf\limits_{x\in H, |x|\to\infty}u(x)\geq0$, then $u\geq0$ in $H$.
\end{lemma}
\begin{proof}
Let $c_{\infty}>0$ be given and denote $H'=\{x\in \R^d\;:\; \dist(x,H)>\delta\}$ for $\delta>0$. Note that we can fix $\delta>0$ such that 
\[
\inf_{x\in \R^d,\ \dist(x,H')<2\delta}\ \int_{H'}j(|x-y|)\ dy>c_{\infty}
\]
since $j\notin L^1(\R^d)$ and hence the above value convergence to infinity for $\delta\to 0$. Next assume that $c,K,u$ are given as stated and that $u$ changes sign in $H$. Moreover, let $v$ be given in \eqref{defi:function-for-anti-mp} and note that $v$ satisfies in viscosity sense
 \[
 \Psi(-\Delta)v\leq c(x)u=-c(x)v\leq c^-(x)v\leq c_{\infty}1_{\Omega\setminus K}v\quad\text{ in $\Sigma:=\{u<0\}\cap \Omega$},
 \]
with $v=0$ in $\R^d\setminus \Sigma$. Moreover, since $\liminf\limits_{x\in H, |x|\to\infty}u(x)\geq0$, we have $\lim\limits_{|x|\to\infty}v(x)=0$ and $v\in \cC_b(\R^d)$. In particular, there is $x\in \Sigma$ with $a:=\max_{\R^d}v=v(x)>0$. Moreover, let $\phi\in \cC^2(\R^d)$ ($0\leq \phi\leq a$) with $\phi\equiv a$ on $H$ and $\phi\equiv 0$ on $H'$. Then $\phi$ touches $v$ from above at $x$ and $\phi\geq v$ in $\R^d$. Thus
\begin{align*}
c_{\infty}1_{\Omega\setminus K}(x)v(x)&=c_{\infty}1_{\Omega\setminus K}(x)a\geq \Psi(-\Delta)\phi(x)=p.v.\int_{\R^d}(\phi(x)-\phi(y))j(|x-y|)\ dy\\
&\geq a\int_{H'}j(|x-y|)\ dy.
\end{align*}
That is, at the maximal point $x$ we have
\[
c_{\infty}1_{\Omega\setminus K}(x)\geq\int_{H'}j(|x-y|)\ dy.
\]
Clearly, from this inequality it follows that we must have $x\in \Omega\setminus K$. Moreover, since $u\geq 0$ in $H\setminus \Omega$ and $\Big(\Omega\cap\{x\in H\;:\; \dist(x,\partial H)\geq \delta\}\Big)\setminus K$ we must have $\dist(x,\partial H)\leq \delta$. But then this is again a contradiction by the choice of $\delta$. Hence $v=0$, and thus $u\geq 0$ in $H$ as claimed.
\end{proof}

\subsection{Proof of Theorem \ref{thm2.1}}

\begin{proof}[Proof of Theorem \ref{thm2.1}]
We follow the idea of moving planes described in the classical case by Serrin \cite{S71} and for the fractional Laplacian, i.e. the case $\Psi(r)=r^{\alpha/2}$, $\alpha\in(0,2)$ in \cite{FJ13}. In the following let $D\subset \R^d$ be a fixed open bounded set with $\cC^2$ boundary and let $u$ be a solution of \eqref{eq:overdetermined} satisfying $\tr_V(u)=c$ on $\partial D$. Note that $u>0$ in $D$ by the maximum principle. Moreover, as explained above, $u$ is H\"older continuous in $\R^d$ \cite{KKLL18} and thus in particular bounded. Given $\lambda\in \R$, $e\in \partial B_1(0)$ denote
\[
v(x)=v_{\lambda,e}(x)=u(x)-u(\bar{x}),\quad x\in \R^d,
\]
where $\bar{x}:=R_{\lambda,e}(x):=x-2(x\cdot e)e+2\lambda e$ denotes the reflection of $x$ at $T_{\lambda,e}:=\partial H_{\lambda,e}$, $H_{\lambda,e}:=\{x\in \R^d\;:\; x\cdot e>\lambda\}$. Note that we have $\R^d\setminus H_{\lambda,e}=H_{-\lambda,-e}$. Moreover, fix $e\in \partial B_1(0)$ and let $\lambda<l:=\sup_{x\in D}x\cdot e$. Then $H\cap D$ is nonempty for all $\lambda<l$ and we put $D_{\lambda}:=R_{\lambda,e}(D\cap H_{\lambda})$. Then for all $\lambda<l$ the function $v$ satisfies in viscosity sense
\begin{equation*}
\left\{\begin{aligned}
\Psi(-\Delta)v&=0&&\text{in $D_{\lambda}$;}\\
v&\geq 0 &&\text{in $H_{-\lambda,-e}\setminus D_{\lambda}$;}\\
v(x)&=-v(\bar{x})&&\text{ for $x\in \R^d$.} 
\end{aligned}\right.
\end{equation*}
Hence we must have $v>0$ in $D_{\lambda}$ or $v\equiv 0$ in $\R^d$ by Theorem~\ref{thm:anti-mp} and \ref{thm:anti-mp2}
for $\lambda$ close to $l$. As we decrease $\lambda$, there are two possible situations that may occur:
\begin{align}
&\text{Situation 1: There is $p_0\in \partial D\cap \partial D_{\lambda}\setminus T_{\lambda,e}$ or}\label{sit1}\\
&\text{Situation 2: $T_{\lambda,e}$ is orthogonal to $\partial D$ at some point $p_0\in \partial D\cap T_{\lambda,e}$.}\label{sit2}
\end{align}

\begin{center}
    	\fbox{
    			\begin{picture}(170,170)
    			\put(15,0){\includegraphics[width=0.3\textwidth,height=0.26\textheight]{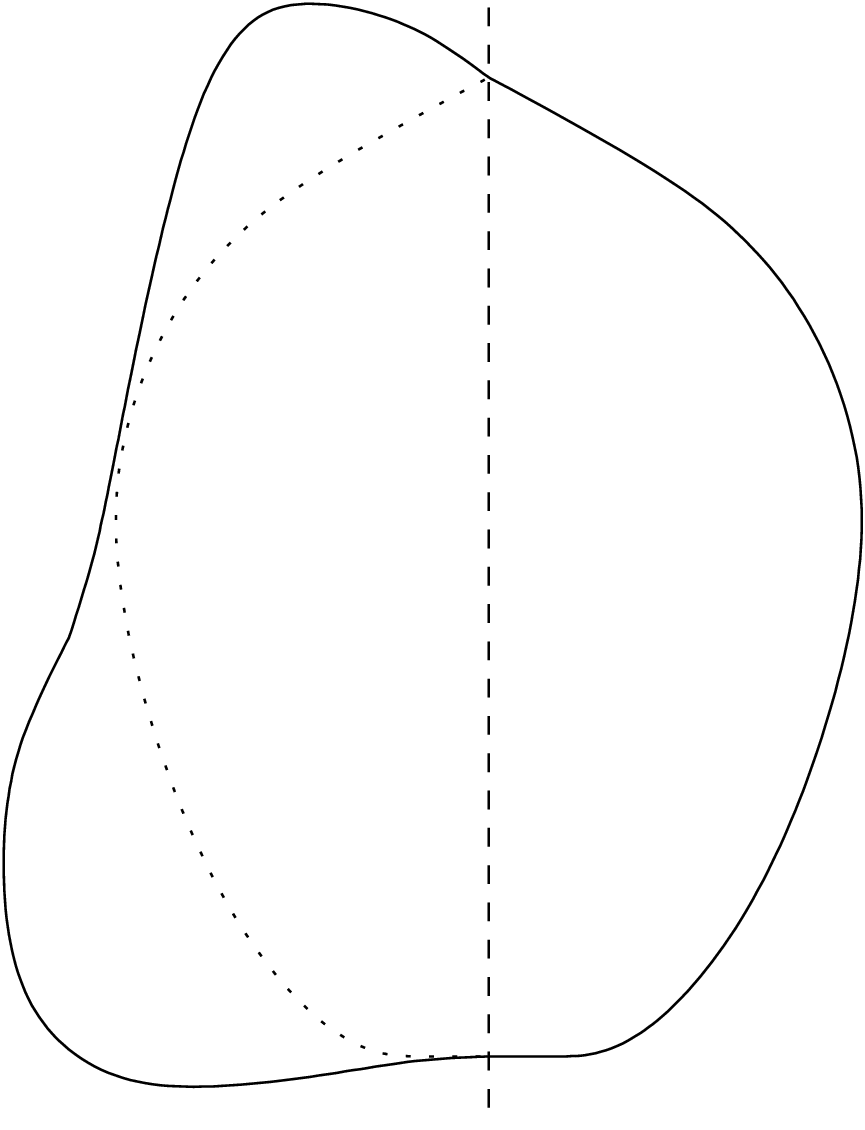}}
    			\put(31.5,97){\huge$\cdot$}
    			\put(19,97){$p_0$}
    			\put(98,1){$\tilde{p}_0$}
    			\put(92,5.5){\huge$\cdot$}
					\put(130,150){$H_{\lambda}$}
					\put(65,80){$D_{\lambda}$}
    			\end{picture}}
    		\captionof{figure}{Situation 1 at $p_0$; Situation 2 at $\tilde{p}_0$.}
    	\end{center}
We fix $\lambda_0$ as the maximal value in $(-\infty,l)$ such that one of these situations occur (or, equivalently, the first time while moving $\lambda$ from $l$ to $-\infty$ where one of the two situations occur). Our goal is to show that in either case we have that $D$ is symmetric with respect to $T_{\lambda_0,e}$, which implies the theorem since $e$ was chosen arbitrarily. For this, we show first that we have $v=v_{\lambda_0,e}\equiv 0$ on $\R^d$. In the following we assume $v>0$ in $D_{\lambda_0}$.\\

\textit{Situation 1:} Note that $u(p_0)=0=u(\bar{p}_0)$ and hence $v(p_0)=0$. Hopf's Lemma, Theorem \ref{thm:anti-mp2}, implies 
$0\neq\tr_V(v)(p_0)=\tr_V(u)-\tr_V(u\circ R_{\lambda_0,e})=0$. This is a contradiction and hence we cannot be in Situation 1.

\vspace{.2in}

\textit{Situation 2:} Let $T=T_{\lambda_0,e}$, $H=H_{\lambda_0,e}$, and $R=R_{\lambda_0,e}$. Moreover, let $p_0\in T\cap \partial D$ such that $T$ is orthogonal to $\partial D$ at $p_0$. By translation and rotation, we may assume $\lambda_0=0$, $p_0=0$, $e=e_1$, and $e_2\in T$ is the interior normal at $\partial D$. Note that this implies $\nabla^2\delta_D(0)$ is diagonal. We have
\begin{lemma}\label{lemma:diagonal-decay}
We have
\[
v(t\bar{\eta})=o(V(t)t) \quad\text{as $t\to 0^+$,}
\]
where $\bar{\eta}=e_2-e_1=(-1,1,0,\ldots,0)\in \R^d$.
\end{lemma}

\begin{lemma}\label{lemma:cornerpoint-hopf}
Let $\Omega\subset \R^d$, $d\geq 2$, be an open set such that $0\in \partial \Omega$ and $\{x_1=0\}$ is orthogonal to $\partial \Omega$ at $0$. Moreover, let $\Omega$ be symmetric about the hyperplane $\{x_1=0\}$ and  there is a ball $B\subset\Omega$ with $\overline{B}\cap \partial\overline{\Omega}=\{0\}$. Denote $D^{\ast}:=\Omega\cap \{x_1<0\}$ and assume $w\in \cC_{b}(\R^d)$ satisfies
\[
\left\{\begin{aligned}
\Psi(-\Delta)w+c(x)w&\geq 0 &&\text{in $D^{\ast}$,}\\
w&\geq 0&&\text{in $\{x_1<0\}$,}\\
w&>0 &&\text{in $D^{\ast}$;}\\
w(x)=w(x_1,x')&=-w(-x_1,x')&&\text{ for $x=(x_1,x')$, $x_1\in \R$, $x'\in \R^{d-1}$}
\end{aligned}\right.
\]
in viscosity sense, where $c\in L^\infty(D^{\ast})$. Let $\bar{\eta}=e_2-e_1=(-1,1,0,\ldots,0)\in \R^d$, then there is $C,t_0>0,$ depending on 
$D^{\ast}$, $N$, $\Psi$, such that
\[
w(t\bar{\eta})\geq C V(t)t\quad\text{ for all $t\in(0,t_0)$.}
\]
\end{lemma}

Combining Lemma \ref{lemma:diagonal-decay} and \ref{lemma:cornerpoint-hopf} (see proofs below), with $\Omega=D\cap \overline{H}\cup D_{\lambda_0}$ and $D^{\ast}=D_{\lambda_0}$, it follows that we cannot have $v>0$ in 
Situation 2.

Thus, we must have $v\equiv 0$ in $\R^d$. From here, the claim follows analogously to the proof of Theorem 1.1 in \cite{FJ13} (see page 9 there).
\end{proof}

\begin{proof}[Proof of Lemma \ref{lemma:diagonal-decay}]
The proof is along the lines of \cite[Lemma 4.1]{JKS25}. We use \cite[Theorem 2.2]{KKLL18}, which gives a function $\psi\in C^{\alpha}(\overline{D})$ for some $\alpha>0$ such that
\[
u(x)=V(\delta_D(x))\psi(x)\quad\text{ for $x\in \R^d$.}
\]
Moreover, by assumption
\[
\psi(x)=c\quad\text{ for $x\in \partial D$.}
\]
Let $\bar{\eta}=e_2-e_1$, $\eta=e_2+e_1$ and in the following we let  $t$ to be small enough such that $t\bar{\eta},t\eta\in D$. Then
\[
\psi(t\bar{\eta})=c+o(1)=\psi(t\eta)\quad \text{ as $t\to 0^+$.}
\]
But then, using \eqref{corrected assumption}, we have
\begin{equation}\label{lemma:diagonal-decay:eq1}
	\begin{split}
		v(t\bar{\eta})=u(t\bar{\eta})-\bar{u}(t\bar{\eta})&=[V(\delta_D(t\bar{\eta}))-V(\delta_D(t\eta))]\psi(t\bar{\eta}) +V(\delta_D(t\eta))[\psi(t\bar{\eta})-\psi(t\eta)]\\
		&=[V(\delta_D(t\bar{\eta}))-V(\delta_D(t\eta))](c+o(1))+o(V(t)t),
		\quad\text{ as $t\to 0$.}
	\end{split}
\end{equation}
Moreover, we have
\begin{align*}
\delta_D(t\bar{\eta})&=\delta_D(0)+ t\nabla \delta_D(0)\cdot\bar{\eta}+\frac{t^2}{2}\nabla^2\delta_D(0)[\bar{\eta}]\cdot\bar{\eta}+o(t^2)\\
&=t e_2\cdot\bar{\eta}+\frac{t^2}{2}\nabla^2\delta_D(0)[\bar{\eta}]\cdot\bar{\eta}+o(t^2)=t+\frac{t^2}{2}C+o(t^2)\quad\text{as $t\to0^+$,}
\end{align*}
where $C=\nabla^2\delta_D(0)[e_2]\cdot e_2+\nabla^2\delta_D(0)[e_1]\cdot e_1$, and similarly
\[
\delta_D(t \eta)=t+\frac{t^2}{2}C+o(t^2)\quad\text{as $t\to0^+$.}
\]
Thus we have $\delta_D(t\bar{\eta})-\delta_D(t\eta)=o(t^2)$ for $t\to 0^+$ and hence, for some $\tau\in (0,1)$, the mean value theorem gives
\begin{align*}
V(\delta_D(t\bar{\eta}))-V(\delta_D(t\eta))&=V'\Big(\delta_D(t\eta)+\tau (\delta_D(t\bar{\eta})-\delta_D(t\eta))\Big)\Big(\delta_D(t\bar{\eta})-\delta_D(t\eta)\Big)\\
&=o(V'(t)t^2)=o(V(t)t)\quad \text{ as $t\to 0^+$.}
\end{align*}
where in the last step we used \cite[Proposition 3.1]{GKK15} (see also \cite[Lemma 2.5]{KKLL18}). 
Combining this with \eqref{lemma:diagonal-decay:eq1} the claim follows.
\end{proof}

\begin{proof}[Proof of Lemma \ref{lemma:cornerpoint-hopf}]
By assumption, we can fix a ball $B=B_R(Re_2)\subset \Omega$ for some $R>0$ small enough with $\partial B\cap \partial \Omega=\{0\}$. We put
\[
K:=B\cap\{x_1<0\}.
\]
\begin{center}
\fbox{
    			\begin{picture}(170,170)
    			\put(20,0){\includegraphics[width=0.3\textwidth,height=0.26\textheight]{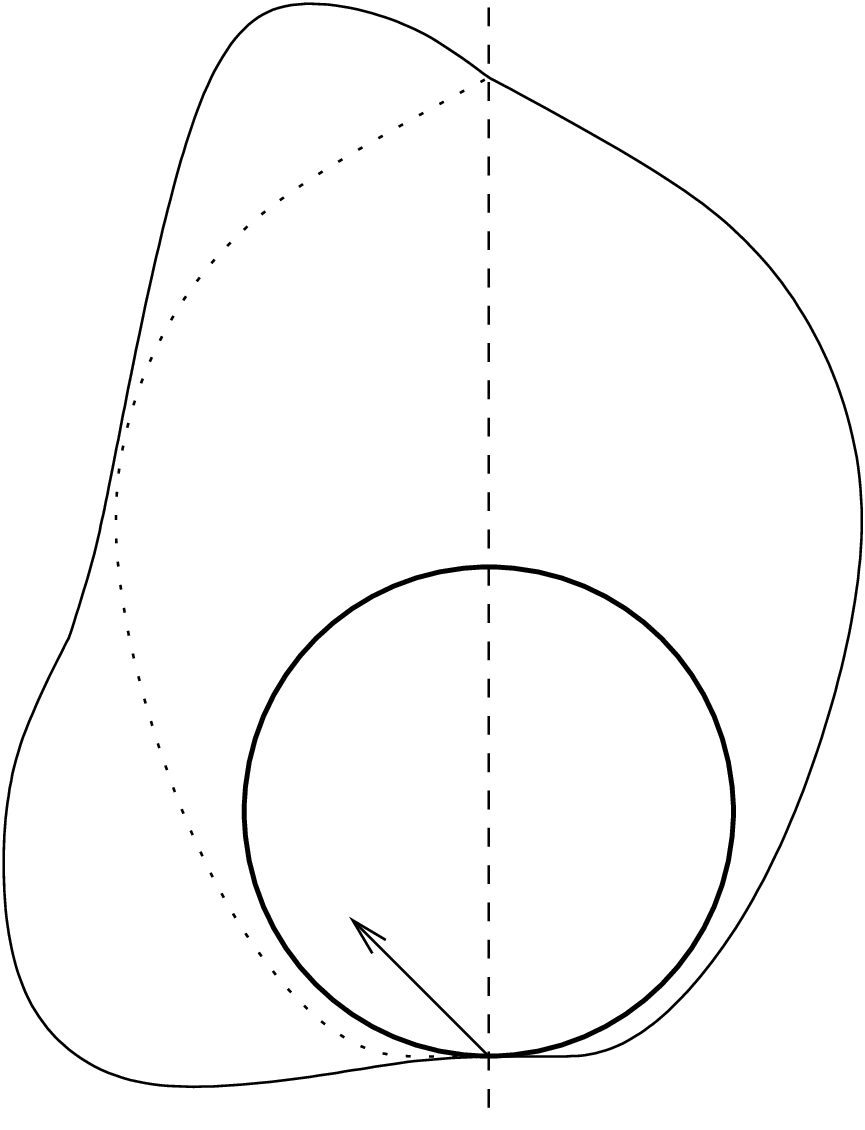}}
    			\put(103,1){$0$}
    			\put(96.85,5.7){\huge$\cdot$}
					\put(130,150){$\{x_1>0\}$}
					\put(59,85){$D^{\ast}$}
					\put(75,55){$K$}
					\put(96.85,40){\huge$\cdot$}
					\put(74,35){$\bar{\eta}$}
    			\end{picture}}
    	\end{center}
Moreover, let $M_1\Subset D^{\ast}$ such that $\delta=\underset{M_1}{\essinf}\ w>0$ and let $M_2=R_{0,e_1}(M_1)$. Note that we may assume that $M_1$ is an open ball. Moreover, by making $R$ smaller, we may assume that $\dist(M_1,K)>0$. Moreover, by making $R$ even smaller if necessary, we may also assume that $|K|$ is small enough with respect to $\|c\|_{L^{\infty}(\Omega)}$ to apply the second assertion of Theorem \ref{thm:anti-mp}. Let $g$ be the unique (viscosity) solution of 
\[
\left\{\begin{aligned}
\Psi(-\Delta)g&=1&&\text{in $B$;}\\
g&=0 &&\text{in $\R^d\setminus B$.}\\
\end{aligned}\right.
\]
Note that $g(x)=u_R(x-e_2)$ with $u_R$ satisfying \eqref{eq:ball} from the introduction. Let $\phi\in C^{\infty}_c(\R^d)$ with $\supp\ \phi\subset M_1$, $0\leq \phi\leq 1$ and there is $U\Subset M_1$, $|U|>0$ such that $\phi=1$ in $U$. We claim that there is $C>0$ such that the bounded continuous function
\[
x\mapsto h(x)=-\kappa x_1g(x)+\delta \phi(x)-\delta \phi(R_{0,e_1}(x))
\]
satisfies for some fixed $\kappa>0$, to be chosen later,
\begin{equation}\label{lemma:cornerpoint-hopf:eq2}
\Psi(-\Delta)h+c(x)h\leq 0\quad \text{ in $K$},
\end{equation}
in viscosity sense. Having shown this, and noting that by construction $w-h$ is anti-symmetric bounded and continuous with $w-h\geq 0$ on $\{x_1<0\}\setminus K$, the maximum principle for anti-symmetric functions, Theorem \ref{thm:anti-mp} (in $K$), implies
\[
w(t\bar{\eta})\geq h(t\bar{\eta})=\kappa tV(t)\quad \text{for $t>0$ small enough,}
\]
as claimed. It remains to show \eqref{lemma:cornerpoint-hopf:eq2}. For this, let $x\in K$ and we write
\[
\Psi(-\Delta)h(x)=-\kappa\Psi(-\Delta)(x_1g)(x)+ \delta \Psi(-\Delta)[\phi-\phi \circ R_{0,e_1}](x).
\]
For the second term (which can be computed classically, since it is smooth in $K$) note that 
\begin{align*}
\Psi(-\Delta)[\phi-\phi \circ R_{0,e_1}](x)&=-\int_{M_1}\phi(y)j(|x-y|)-j(|x-R_{0,e_1}(y)|)\ dy\\
&\leq -\int_{U}\Big(j(|x-y|)-j(|R_{0,e_1}(x)-y|)\ dy\leq -C_1,
\end{align*}
where $C_1>0$ is a constant depending only on $K$, $U$, and $j$. And for the first term we show that
for some constant $C_2$ we have
\begin{equation}\label{A10}
\Psidel (x_1g(x))\leq C_2\quad \text{in}\; K,
\end{equation}
in viscosity sense. Let $x$ be a point in $K$ and $\psi$ be a $\cC^2$ function that touches $-y_1 g(y)$ from above at $x$. Let $B_x\subset B_1(0)$ be 
a ball around $x$ satisfying
$B_x\Subset K$. Let
\[
\Phi(y)=\left\{\begin{array}{ll}
\psi(y) & \text{for}\; y\in B_x,
\\
-y_1g(y) & \text{otherwise}.
\end{array}
\right.
\]
Thus to establish \eqref{A10} we need to show that
$$\Psidel \Phi(x) \leq C_2.$$
Let us also define 
\[
\widehat\Phi(y)=\left\{\begin{array}{ll}
-\frac{1}{y_1}\psi(y) & \text{for}\; y\in B_x,
\\
g(y) & \text{otherwise}.
\end{array}
\right.
\]
Now by \cite[Theorem 1.1]{KKLL18} we have 
\[
|g(y)-g(x)|\leq C_4 \chi(|x-y|)\quad\text{ for $x,y\in \R^d$}, \quad \text{where}\quad \chi(r) = \Psi(r^{-2})^{-\frac{1}{2}},
\]
and $C_4>0$ depending on $d$, $B$, and $\Psi$. Observe that, by \eqref{assumption1}, $\chi(r)\leq C_5 r^{a_1}$ for all $r<1$. Thus 
$g$ is $a_1$-H\"{o}lder continuous in $\R^d$.  Let 
\[
\Phi'(y)=\left\{\begin{array}{ll}
(-\frac{1}{y_1}\psi(y))\wedge(g(x) + C_4 \chi(|x-y|)) & \text{for}\; y\in B_x
\\
g(y) & \text{otherwise}.
\end{array}
\right.
\]
Note that $\Phi'$ touches $g$ from above at $x$ and when $y$ is very close to $x$ we have $\widehat\Phi(y)=\Phi'(y)$. This is possible 
since $\chi(r)\gtrsim r^{a_2}$ where $a_2<1$.
Since $g$ is a viscosity solution we have $\Psidel \Phi'(x)\leq 1$. On the other hand $\Phi(y)\geq -y_1\Phi'(y):=\Phi{''}(y)$ in $\R^d$.
Let us now compute

\begin{align*}
\Psi(-\Delta)(\Phi)(x)&\leq \Psi(-\Delta)(\Phi^{''})(x)\\
&=\int_{\R^d}(-x_1\Phi^{'}(x)+(x_1+y_1)\Phi^{'}(x+y)-1_{\{|y|\leq 1\}}y\cdot \nabla [x_1\Phi^{'}(x)])j(|y|)\ dy\\
&=-x_1\Psi(-\Delta)\Phi^{'} (x)+\int_{\R^d}(y_1\Phi^{'}(x+y)-1_{\{|y|\leq 1\}}y_1\Phi^{'}(x))j(|y|)\ dy\\
&\leq -x_1+\int_{B_1(0)}(\Phi^{'}(x+y)-\Phi^{'}(x))y_1j(|y|)\ dy + \int_{\R^d\setminus B_1(0)} y_1 g(x+y)j(|y|)\ dy\\
&\leq C_3+ \int_{B_1(0)}(\Phi^{'}(x+y)-g(x))y_1j(|y|)\ dy\ dy+C_3,
\end{align*}
Since $\Phi'$ touches $g$ from above it follows that $|\Phi^{'}(x+y)-g(x)|\leq C_4 \chi(|y|)$ and therefore,
using \eqref{J-V} we have
\begin{align*}
&\int_{B_1(0)}(g(x+y)-g(x))|y| j(|y|)\ dy\leq C_5 \int_{B_1(0)} \Psi(|y|^{-2})^{\frac{1}{2}}|y|^{1-d}\ dy\leq C_5\int_0^{1} \Psi(r^{-2})^{\frac{1}{2}}\ dr=C_6,
\end{align*}
where $C_5, C_6>0$ is a constant depending on $d$, $B$, and $\Psi$ and the finiteness follows from \eqref{assumption1} and the fact that $a_2<1$. Hence \eqref{A10} holds. Next, let $\kappa\leq \frac{\delta C_1}{C_2 +\|c\|_{L^{\infty}(D^{\ast})} \|g\|_{L^{\infty}(B)}}$, then for $x\in K$ we have
\begin{align*}
\Psi(-\Delta)h(x)\leq \kappa C_2-\delta C_1\leq -\|c\|_{L^{\infty}(D^{\ast})}\kappa \|g\|_{L^{\infty}(B)}\leq -c(x)h(x)
\end{align*}
as claimed in \eqref{lemma:cornerpoint-hopf:eq2}.
\end{proof}

\subsection{Proof of Theorem \ref{thm2.2}}

\begin{proof}[Proof of Theorem \ref{thm2.2}]
As in the proof of Theorem \ref{thm2.1}, given $\lambda\in \R$, $e\in \partial B_1(0)$ denote
\[
v(x)=v_{\lambda,e}(x)=u(x)-u(\bar{x}),\quad x\in \R^d,
\]
where $\bar{x}:=R_{\lambda,e}(x):=x-2(x\cdot e)e+2\lambda e$ denotes the reflection of $x$ at $T_{\lambda,e}:=\partial H_{\lambda,e}$, $H_{\lambda,e}:=\{x\in \R^d\;:\; x\cdot e>\lambda\}$.  Moreover, fix $e\in \partial B_1(0)$ and let $\lambda<l:=\sup_{x\in D}x\cdot e$. Then $H_{\lambda,e}\cap D$ is nonempty for all $\lambda<l$ and we put $D_{\lambda}:=R_{\lambda,e}(D\cap H_{\lambda})\subset H_{-\lambda,-e}$. Then for all $\lambda<l$ the function $v$ satisfies in viscosity sense
\[
\left\{\begin{aligned}
\Psi(-\Delta)v+c(x)v&=0&&\text{in $D_{\lambda}$;}\\
v&\geq 0 &&\text{in $H_{\lambda,-e}\setminus D_{\lambda}$;}\\
v(x)&=-v(\bar{x})&&\text{ for $x\in \R^d$,} 
\end{aligned}\right.
\]
where 
\[
c(x)=c_f(x)=\left\{\begin{aligned} &\frac{f(u(\bar{x})-f(u(x))}{u(x)-u(\bar{x})}, &&u(x)\neq u(\bar{x});\\
& 0,&& u(x)=u(\bar{x}).\end{aligned}\right.
\]
Note that since $u$ is bounded and $f$ is locally Lipschitz continuous, there is $c_{\infty}>0$ independent of $\lambda$ such that
\[
\|c\|_{L^{\infty}(D_{\lambda})}\leq c_{\infty} \quad\text{ for all $\lambda<l$.}
\]
Moreover, since $u$ is continuous, we have hence $c\in L^\infty (D_{\lambda})$. Next, we show that we can move the hyperplanes from $l$ up to the first occurrence of either situation \eqref{sit1} or \eqref{sit2}. Denote $\lambda_0<l$ as this first occurrence. In order to apply the argumentation in these two situations as in the case $f(u)\equiv 1$, we need to show that
\[
v=v_{\lambda}=u-u\circ R_{\lambda,e}\geq 0 \quad\text{ in $D_{\lambda}$ for all $\lambda\in[\lambda_0,l)$.}
\]
First note that due to Theorem \ref{thm:anti-mp} there is $\epsilon>0$ such that we have $v\geq 0$ in $D_{\lambda}$ for $\lambda\in[l-\epsilon,l)$. Let
$$
\lambda^{\ast}=\inf\{\lambda<l\;:\; v_{\mu}\geq0 \text{ in $H_{\mu,-e}$ for all $\mu\in(\lambda,l)$}\}
$$
 Clearly, $\lambda^{\ast}\geq \lambda_0$. Assume by contradiction $\lambda^{\ast}>\lambda_0$. By continuity, we have $v=v_{\lambda^{\ast}}\geq0$ in $D_{\lambda^{\ast}}$. Then Theorem \ref{thm:anti-mp2} implies $v\equiv 0$ on $\R^d$ or $v>0$ in $D_{\lambda^{\ast}}$. Assume that $v\equiv 0$ on $\R^d$. Then $u$ is symmetric with respect to $\partial H_{\lambda^{\ast},e}$ since $D\setminus (H_{\lambda^{\ast},e}\cup D_{\lambda^{\ast}})$ has nonempty interior by assumption, we can also move hyperplanes from $l_-:=\sup_{x\in D}x\cdot (-e)$ to $-\infty$ and due to the symmetry of $u$ with respect to $\partial H_{\lambda^{\ast},e}$ there is $\epsilon'>0$ such that $v_{-\lambda,-e}\equiv 0$ on $D_{-\lambda,-e}$ for $\lambda\in (l_--\epsilon',l_-)$. But then this implies $u\equiv 0$ in contradiction to the assumption that $u$ is nontrivial. Hence we have $v=v_{\lambda^{\ast}}>0$ in $D_{\lambda^{\ast}}$. By continuity of $v$ and $\lambda\mapsto v_{\lambda}$, there is $\delta>0$ and $K\subset D_{\lambda^{\ast}}$ such that for any $\epsilon\in[0,\delta]$ we have $v_{\lambda^{\ast}-\epsilon}>0$ in $K$ and $|D_{\lambda^{\ast}-\delta}\setminus K|$ is small enough (with respect to $c_{\infty}$) to apply the second assertion of Theorem \ref{thm:anti-mp}. But then, Theorem \ref{thm:anti-mp} applied to $\Omega=D_{\lambda^{\ast}-\epsilon}\setminus K$ implies $v_{\lambda^{\ast}-\epsilon}\geq0$ in $H_{\lambda^{\ast}-\epsilon,-e}$ for all $\epsilon\in[0,\delta]$. This is a contradiction to the definition of $\lambda^{\ast}$ and hence we must have $\lambda^{\ast}=\lambda_0$. From here, the proof follows as the proof of Theorem \ref{thm2.1}.
\end{proof}

\subsection{Proof of Theorem \ref{thm2.3}}

In the following, we assume there is a solution as stated in Theorem \ref{thm2.3} and we let $G_1,\ldots,G_n, G$ be compact sets and $A_1,\ldots,A_n,A\in \R$, $A>0$ as stated. We proceed similarly as in the proof of Theorem \ref{thm2.1} and \ref{thm2.2}, however, arguing with the Situations 2 occurring in the previous proofs this time for the set $G$. We continue with the notation of the previous sections. In particular, the function $v_{\lambda}=v_{\lambda,e}=u-u\circ R_{\lambda,e}$ for $\lambda\in \R$, $e\in \partial B_1(0)$ satisfies in viscosity sense
\begin{equation*}
\left\{\begin{aligned}
\Psi(-\Delta)v_{\lambda}+c(x)v_{\lambda}&=0&&\text{in $D_{\lambda}$;}\\
v_{\lambda}&\geq 0 &&\text{in $H_{-\lambda,-e}\setminus D_{\lambda}$;}\\
v_{\lambda}(x)&=-v_{\lambda}(R_{\lambda,e}(x))&&\text{ for $x\in \R^d$,} 
\end{aligned}\right.
\end{equation*}
where $D_{\lambda}=H_{-\lambda,-e}\setminus (G\cup R_{\lambda,e}(G))$ and
\[
c(x)=c_{f,\lambda,e}(x)=\left\{\begin{aligned} &\frac{f(u(R_{\lambda,e}(x))-f(u(x))}{u(x)-u(R_{\lambda,e}x)}, && u(x)\neq u(R_{\lambda,e}(x)),
\\
& 0,&& u(x)=u(R_{\lambda,e}(x)).\end{aligned}\right.
\]
Since $f$ is assumed to be Lipschitz continuous and $u$ is bounded we have 
\[
\|c\|_{L^{\infty}(D_{\lambda})}\leq c_{\infty} \quad\text{ for all $\lambda\in \R$, $e\in \partial B_1(0)$.}
\]
As before, we denote $\lambda_0\in\R$ as the largest number such that we have
\begin{align}
&\text{Situation 1: There is $p_0\in \partial G\cap \partial R_{\lambda,e}(G\cap H_{\lambda,e})\setminus T_{\lambda,e}$ or}\label{sit1-b}\\
&\text{Situation 2: $T_{\lambda,e}$ is orthogonal to $\partial G$ at some point $p_0\in \partial G\cap T_{\lambda,e}$.}\label{sit2-b}
\end{align}
and moreover
\begin{align*}
\lambda^{\ast}:=\inf\{\lambda\geq\lambda_0\;:\; v_{\mu}\geq0 \text{ in $H_{-\mu,-e}$ for all $\mu\in(\lambda,\infty)$}\}
\end{align*} 
Clearly, as in the previous case, we have $\lambda_0\in \R$ due to the boundedness of $G$ and the regularity assumptions on its boundary. In the following, we let $e\in \partial B_1(0)$ be fixed and we aim at showing that we have $\lambda_0=\lambda^{\ast}$ to conclude our result. For simplicity of the notation, we may assume $\lambda_0=0$ and $e=e_1$ by translation and rotation of the problem, and denote 
\begin{align*}
H_{\lambda}&=\{x_1>\lambda\},\\
 H^{\lambda}&=\{x_1<\lambda\}=\R^d\setminus \overline{H_{\lambda}}=H_{-\lambda,-e_1},\\
R_{\lambda}(x)&=R_{\lambda,e_1}(x)=(2\lambda-x_1,x'),\qquad  \text{for $x=(x_1,x')$, $x\in \R$, $x'\in \R^{d-1}$, and }\\
D_{\lambda}&=\{x_1<\lambda\}\setminus (G\cup R_{\lambda}(G)).
\end{align*}
Moreover, we consider the statement
\begin{equation*}
(S_{\lambda})\quad v_{\mu}> 0 \quad\text{ in $D_{\mu}$ for $\mu\geq \lambda$.}
\end{equation*}
Since we assume $f$ is non-increasing for small arguments we can find $f_0\in(0,A]$ such that $f|_{[0,f_0]}$ is non-increasing. Note that since $u\to 0$ for $|x|\to \infty$, given $\lambda>0$ there is $R>0$ such that $|v_{\lambda}|\leq u_{\lambda}+u\leq f_0$ in $\R^d\setminus B_R(0)$ and in particular, $c\geq 0$ 
 on $D_{\lambda}\setminus B_R(0)$. To begin the moving plane method, we note that
\begin{lemma}\label{lemma41}
	$(S_{\lambda})$ holds for $\lambda$ large enough.
\end{lemma}
\begin{proof}
	 We first show that we have $v_{\lambda}\geq0$ in $H^{\lambda}$ for $\lambda$ large enough. First note that $G\subset H^{\lambda}$ for $\lambda$ large and hence $D_{\lambda}=H^{\lambda}\setminus G$. Moreover, since $u\to 0$ for $|x|\to \infty$, we can fix $R>0$ such that $u\leq \frac{f_0}{2}$ (with $f_0\in(0,A]$ such that $f|_{(0,f_0]}$ being non-increasing). Hence, by making $\lambda$ larger if necessary, we have $u\circ R_{\lambda}\leq \frac{f_0}{2}$ in $H^{\lambda}$. Moreover, clearly $v_{\lambda}\geq0$ in $\{u\geq f_0/2\}\cap H^{\lambda}$ and $c\geq0$ in $\Omega:=D_{\lambda}\cap \{u<f_0/2\}$. With $K=\Omega$ in Lemma \ref{lemma:mp-unbounded} implies $v_{\lambda}\geq 0$ in $H^{\lambda}$.
Having shown $v_{\lambda}\geq0$ in $H^{\lambda}$ for $\lambda$ large, we note that Theorem \ref{thm:anti-mp2} implies $v_{\lambda}\equiv 0$ or $v_{\lambda}>0$. But since $u$ is nontrivial and $u\to 0$ for $|x|\to \infty$ we cannot have $v_{\lambda}=0$ for $\lambda$ large. This finishes the proof.
\end{proof}

\begin{lemma}\label{lemma42}
	$(S_{\lambda})$ holds for all $\lambda>\lambda^{\ast}$. Moreover, $v_{\lambda^{\ast}}\geq0$ in $H^{\lambda^{\ast}}$.	
\end{lemma}
\begin{proof}
	Let $\bar{\lambda}=\inf\{\lambda\;:\; (S_\lambda) \text{ holds}\}$. Then $\lambda^{\ast}\leq \bar{\lambda}<\infty$ by Lemma \ref{lemma41}. And by continuity, we have $v_{\bar{\lambda}}\geq0$ in $H^{\bar{\lambda}}$. The statement follows once we have shown $\bar{\lambda}=\lambda^{\ast}$. Assume by contradiction $\bar{\lambda}>\lambda^{\ast}$ and let $\lambda\in(\lambda^{\ast},\bar{\lambda}]$. By Theorem \ref{thm:anti-mp2} we have $v_{\lambda}\equiv 0$ in $\R^d$ or $v_{\lambda}>0$ in $D_{\lambda}$. First, assume $v_{\lambda}\equiv 0$ in $\R^d$. Note that this implies we have $u(t,x')=u(2\lambda-t,x')$ for all $t\in \R$, $x'\in \R^{d-1}$. Moreover, since $v_{\mu}\geq0$ in $H^{\mu}$ for $\mu\geq \lambda^{\ast}$, we have $u(t,x')\geq u(2\mu-t,x')$ for all $t<\mu$, $\mu\geq \lambda^{\ast}$ and $x'\in \R^{d-1}$. In particular, we have $\mu\in(\lambda^{\ast},\lambda)$, $t<\mu$, and $x'\in \R^{d-1}$ we have
	\[
	u(2\lambda-t,x')=u(t,x')\geq u(2\mu-t,x')\geq u(2\lambda-t,x'),
	\]
	where in the last inequality, we have used the monotonicity of $u$ for $\tilde{\mu}=(\mu+\lambda)-t\geq\lambda^{\ast}$---note that $2\mu-t<\tilde{\mu}$ and $R_{\tilde{\mu}}(2\mu-t,x')=(2\lambda-t,x')$. This implies that we have $v_{\mu}\equiv 0$ in $\R^d$ for all $\mu\in[\lambda^{\ast},\lambda]$. Let $\bar{\mu}=\sup\{\mu\;:\; v_{\mu}\equiv 0\}$. Then $\lambda\leq \bar{\mu}\leq \bar{\lambda}$ and by continuity $v_{\bar{\mu}}\equiv 0$ in $\R^d$, but $v_{\bar{\mu}+\epsilon}>0$ in $H^{\bar{\mu}+\epsilon}$ for all $\epsilon>0$. But this implies $t\mapsto u(t+\bar{\mu},x')=u(-t+\bar{\mu},x')$ is strictly decreasing on $(0,\infty)$ for all $x'\in \R^{d-1}$, which is a contradiction to the symmetry of $u$ with respect to $\{x_1=\mu\}$ for all $\mu\in(\lambda^{\ast},\lambda)$. 	Hence we must have $v_{\lambda}>0$ in $D_{\lambda}$. But since $\lambda$ was chosen arbitrarily in $(\lambda^{\ast},\bar{\lambda}]$, this contradicts the definition of $\bar{\lambda}$ and hence we must have $\bar{\lambda}=\lambda^{\ast}$.
\end{proof}

\begin{lemma}\label{lemma43}
	If $\lambda^{\ast}>\lambda_0$, then $v_{\lambda^{\ast}}\equiv 0$ on $\R^d$.
\end{lemma}
\begin{proof}
	Note that Theorem \ref{thm:anti-mp2} implies either the claim or $v_{\lambda^{\ast}}>0$ in $D_{\lambda^{\ast}}$. Hence, we may assume by contradiction that the latter holds. Recall that we assume $f|_{[0,f_0]}$ is non-increasing for some $f_0\in(0,A]$. So we choose $R$ large enough so that 
	$|v_{\lambda^{\ast}}|\leq \frac{1}{2}f_0$ on $\R^d\setminus B_R(0)$. For any $\delta\in (0, \lambda_0-\lambda^*)$ we define $U=B_R(0)\cap D_{\lambda^{\ast}}\cap \{x_1\leq \lambda^{\ast}-\delta\}$. We claim that $\min_{\bar U}v_{\lambda^{\ast}}=\eps>0$. If this does not hold true, $\bar{U}$ being compact, we can find a point 
	$x\in \bar{D}_{\lambda^{\ast}}\cap G\cap\{x_1\leq \lambda^{\ast}-\delta\}$ such that $v_{\lambda^{\ast}}(x)=0$. But, since $u<A$ in $G^c$, this is possible if
	$x\in R_{\lambda^*} (G\cap H_{\lambda^*})\cap \partial(G\cap H^{\lambda^*})\setminus\{x_1=\lambda^*\}$. This is contradicting to the definition of $\lambda_0$. Thus we must have 
	$\min_{\bar U}v_{\lambda^{\ast}}=\eps>0$. Choose $\mu\in (\lambda^*, \lambda_0)$. Using continuity we note that for $\mu$ sufficiently close to $\lambda^*$ we must have $v_{\mu}>0$ in $\bar{U}$. Moreover, on $H^{\mu}\setminus D_{\mu}$ we have $v_{\mu}\geq0$, where we use again $\lambda^{\ast}> \lambda_0$.
	Lemma \ref{lemma:mp-unbounded} applied to $\Omega= D_{\mu}\setminus U$, $K=D_{\mu}\setminus B_R(0)$ implies $v_{\mu}\geq 0$ in $H^\mu$ for $\mu<\lambda^{\ast}$, $\lambda^{\ast}-\mu$ small, which is a contradiction to the definition of $\lambda^{\ast}$.  Hence $v_{\lambda^{\ast}}\equiv 0$ as claimed.
\end{proof}

\begin{lemma}
	\label{lemma44}
	We have $\lambda^{\ast}=\lambda_0$.
\end{lemma}
\begin{proof}
Recall, $\lambda_0=0\leq \lambda^{\ast}$. Assume by contradiction $\lambda^{\ast}>0$ and note that by Lemma \ref{lemma42} and \ref{lemma43} we have that $u$ is symmetric with respect to $\partial H_{\lambda^{\ast}}$ and strictly decreasing in $x_1>\lambda^{\ast}$, i.e. for $t>s\geq\lambda^{\ast}$ and $x'\in \R^{d-1}$ we have $u(t,x')<u(s,x')$ and $u(s,x')=u(2\lambda^{\ast}-s,x')$. Moreover, there is $x_0\in G\cap H^{\lambda^{\ast}}\setminus R_{\lambda^{\ast}}(G)$ since $\lambda^{\ast}>0$.
 Since $u(x_0)=A$ is a global maximum this contradicts the fact that $u$ is strictly decreasing in the direction $x_1>\lambda^{\ast}$ (this follows similarly as in the proof of Lemma \ref{lemma42}). Hence $\lambda^{\ast}=0$ as claimed. 
\end{proof}

\begin{remark}
	We note that due to Lemma \ref{lemma41}--Lemma \ref{lemma44} and by rotation and translation, we actually have $u$ is symmetric about $\partial H_{\lambda_0,e}$ and strictly decreasing in the direction $x\cdot e$ away from $\partial H_{\lambda_0,e}$. In particular, this implies that $G$ must be connected in $x\cdot e$ direction and symmetric about $\partial H_{\lambda_0,e}$.
\end{remark}

\begin{proof}[Proof of Theorem \ref{thm2.3}]
	The Lemma \ref{lemma41}--\ref{lemma44} imply that we can move the hyperplanes up the first time one of the two situations \eqref{sit1-b} and \eqref{sit2-b} occurs. From there, the statement follows as outlined in the proof of Theorem \ref{thm2.1}. Indeed, Theorem \ref{thm:anti-mp2} rules out $v_{\lambda_0}>0$ in $D_{\lambda_0}$ in Situation 1, while in Situation 2 we can rule out $v_{\lambda_0}>0$ in $D_{\lambda_0}$ by a combination of Lemma \ref{lemma:diagonal-decay} and Lemma \ref{lemma:cornerpoint-hopf}. For the application of the first Lemma, we note that also in the unbounded case we indeed have $u/[V\circ \delta_{D}]$ is H\"older continuous (see Remark \ref{R2.2}).	It then follows that we have $v_{\lambda_0}\equiv 0$ and then the statement follows analogously to the proof of Theorem \ref{thm2.1} (see also the argument for the end of the proof of \cite[Theorem 1.3]{SV14}).
\end{proof}

To extend the proof of Theorem \ref{thm2.3} to the situation of Theorem \ref{thm2.3a} we need the following adjustment of Lemma \ref{lemma43} and Lemma \ref{lemma44}. 

\begin{lemma}\label{lemma43a}
Assuming $f$ is nonincreasing on $[0,A]$, but $0\leq u\leq A$, then $\lambda^{\ast}=\lambda_0$.
\end{lemma}
\begin{proof}
Suppose that $\lambda^{\ast}>\lambda_0$. Then it follows that none of the situations \eqref{sit1-b}, \eqref{sit2-b} can occur of $\lambda>\lambda^*-\eps$ for some $\eps>0$ small. Fix $\lambda\in(\lambda^{\ast}-\epsilon,\lambda^{\ast})$ and note that we have $R_{\lambda, e}(H_{\lambda}\cap G)=R_{\lambda, e}(G)\cap H^{\lambda}\subset H^{\lambda}\cap G$. Thus $v_\lambda \geq 0$ in $H^{\lambda}\cap G$. Moreover, $v_{\lambda}$ solves
\[
\Psi(-\Delta)v_{\lambda} +c(x)v_{\lambda}\geq0\quad\text{ in $D_{\lambda}=H^{\lambda}\setminus G$.} 
\]
Note that since $f$ is nonincreasing, we have $c\geq0$ on $D_{\lambda}$. Since also $v_{\lambda}=v_{\lambda}\circ R_{\lambda}$ on $T_{\lambda, e}$ and $v_{\lambda}\to 0$ for $|x|\to \infty$, Lemma \ref{lemma:mp-unbounded} with $K=\Omega=D_{\lambda}$ implies $v_{\lambda}\geq0$ in $H^{\lambda}$. Since $\lambda$ was chosen arbitrarily in $(\lambda^{\ast}-\epsilon,\lambda^{\ast})$, this gives a contradiction to the definition of $\lambda^{\ast}$ and hence we must have $\lambda^{\ast}=\lambda_0$.
\end{proof}

\begin{proof}[Proof of Theorem \ref{thm2.3a}]
The proof follows as the proof of Theorem \ref{thm2.3} but with Lemma \ref{lemma43a} in place of Lemma \ref{lemma43} and Lemma \ref{lemma44}.
\end{proof}

\subsection{Proof of Theorem \ref{thm2.4} and Corollary \ref{cor-from-thm2.4}}

\begin{proof}[Proof of Theorem \ref{thm2.4}]
Analogously to the Lemmas \ref{lemma41}--\ref{lemma44} (with $\lambda_0=0$), using the same notation, we have for any $e\in \partial B_1(0)$ that $\lambda^{\ast}=0$, $v_{0,e}\equiv 0$ and $v_{\lambda,e}>0$ in $H_{-\lambda,-e}$ for all $\lambda>0$. We do not need $u<A$, since $G=\{0\}\subset H_{-\lambda,-e}$ for all $\lambda\geq0$. These statements imply that for each $e\in \partial B_1(0)$, $\lambda\in \R$ we have $u\geq u\circ R_{\lambda,e}$ on $H_{\lambda,e}$ or $u\leq u\circ R_{\lambda,e}$ on $H_{\lambda,e}$. \cite[Proposition 6.1]{FJ13} then implies $u$ is radially symmetric and strictly decreasing (due to Lemma \ref{lemma42}).
\end{proof}

\begin{proof}[Proof of Corollary \ref{cor-from-thm2.4}]
If $u$ is nontrivial, there is, by assumption, $A=\max_{\R^d}u=u(x_0)>0$ for some $x_0\in \R^d$. Replacing $u$ with $u(\cdot-x_0)$ the statement follows from Theorem \ref{thm2.4}.	
\end{proof}

\subsection{Proofs of Theorem \ref{thm2.3b} and Corollaries \ref{cor-from-thm2.3} and \ref{cor-from-thm2.3b}}

\begin{proof}[Proof of Theorem \ref{thm2.3b}]
	This statement follows analogously to the proof of Theorem \ref{thm2.3} (see also the proof of Theorem \ref{thm2.4}). However, since $D$ is bounded, the moving plane argument can be done as prescribed in the proof of Theorem \ref{thm2.2} using Theorem \ref{thm:anti-mp} instead of Lemma \ref{lemma:mp-unbounded}. In fact, one can follow the arguments in \cite[Theorem~1.7]{SV14}.
\end{proof}

\begin{proof}[Proof of Corollary \ref{cor-from-thm2.3} and \ref{cor-from-thm2.3b}]
	Both statements follow immediately from the moving plane argument explained in the proof of Theorem \ref{thm2.2}, \ref{thm2.3}, and \ref{thm2.4}.
\end{proof}

\subsection*{Acknowlodgements}
This research of AB was supported in part by an INSPIRE faculty fellowship and DST-SERB grants
 EMR/2016/004810, MTR/2018/000028. 


\end{document}